\documentclass[12pt]{amsart}
\usepackage{amssymb}
\usepackage{amsmath, amscd}
\usepackage{amsthm}
\usepackage{amsmath}
\usepackage{comment}
\usepackage[table,xcdraw]{xcolor}
\usepackage{ulem}
\usepackage{tikz}
\usepackage{float}
\usepackage{graphicx}
\usepackage{circuitikz}
\usetikzlibrary{positioning}
\usepackage[colorlinks=true, linkcolor=blue,urlcolor=blue]{hyperref}

\newtheorem{theorem}{Theorem}[section]

\newtheorem{proposition}[theorem]{Proposition}
\newtheorem{lemma}[theorem]{Lemma}

\newtheorem{corollary}[theorem]{Corollary}
\theoremstyle{definition}

\newtheorem{example}[theorem]{Example}
\newtheorem{definition}[theorem]{Definition}

\newtheorem{remark}[theorem]{Remark}

\begin{document}
	
\author[Z. Vesali Mahmood]{Zari Vesali Mahmood}
\address{Department of Mathematics, Tarbiat Modares University, 14115-111 Tehran Jalal AleAhmad Nasr, Iran}
\email{v.zarivesali@modares.ac.ir}

\author[A. Moussavi]{Ahmad Moussavi$^*$}
\address{Department of Mathematics, Tarbiat Modares University, 14115-111 Tehran Jalal AleAhmad Nasr, Iran}
\email{moussavi.a@modares.ac.ir; moussavi.a@gmail.com}

\thanks{$^*$Corresponding author: Ahmad Moussavi, email: moussavi.a@modares.ac.ir; moussavi.a@gmail.com}

\author[P. Danchev]{Peter Danchev}
\address{Institute of Mathematics and Informatics, Bulgarian Academy of Sciences, 1113 Sofia, Bulgaria}
\email{danchev@math.bas.bg; pvdanchev@yahoo.com}

\title[2-$UNJ$ rings]{On Rings with the 2-UNJ Property}

\maketitle




\begin{abstract}
In this paper, we introduce a new class of rings calling them {\it 2-UNJ rings}, which generalize the well-known 2-UJ, 2-UU and UNJ rings. Specifically, a ring $R$ is called 2-UNJ if, for every unit $u$ of $R$, the inclusion                                                                                                                                                                                                                                                                                                                                                                $u^2 \in 1 + Nil(R) + J(R)$ holds, where $Nil(R)$ is the set of nilpotent elements and $J(R)$ is the Jacobson radical of $R$. We show that every 2-UJ, 2-UU or UNJ ring is 2-UNJ, but the converse does {\it not} necessarily hold, and we also provide counter-examples to demonstrate this explicitly. We, moreover, investigate the connections between these rings and other algebraic properties such as being potent, tripotent, regular and exchange rings, respectively. In particular, we thoroughly study some natural extensions, like matrix rings and Morita contexts, obtaining new characterizations that were not addressed in previous works. Furthermore, we establish conditions under which group rings satisfy the 2-UNJ property. These results not only provide a better understanding of the structure of 2-UNJ rings, but also pave the way for future intensive research in this area.
\end{abstract}

\section{Introduction and Basic Concepts}

In this paper, let \( R \) be an associative ring with identity \( 1=1_R \), which is {\it not} necessarily commutative. For such a ring \( R \), we denote by \( U(R) \), \( C(R) \), \( Nil(R) \) and \( Id(R) \) the set of units of $R$, the center of $R$, the set of nilpotent elements and the set of idempotent elements of \( R \), respectively. Moreover, we define \( M_n(R) \) and \( T_n(R) \) to be the rings of all \( n \times n \) matrices and upper triangular \( n \times n \) matrices over \( R \), respectively. Traditionally, a ring \( R \) is said to be {\it abelian} if every idempotent element is central, i.e., \( Id(R) \subseteq C(R) \).

It is a well-established fact that the containment \( 1 + J(R) \subseteq U(R) \) is always true. A ring \( R \) is referred to as a {\it UJ-ring} if the reverse inclusion holds as well, that is, \( U(R) = 1 + J(R) \) (see the former source \cite{D}, and \cite{14} too). Motivated by the work done in \cite{1}, a ring \( R \) is defined as {\it 2-UJ} if, for every \( u \in U(R) \), there exists an element \( j \in J(R) \) such that \( u^2 = 1 + j \). These 2-UJ rings naturally extend the concept of UJ-rings. In \cite{1}, it was demonstrated that in the case of 2-UJ ring the properties of being semi-regular, exchange and clean are all equivalent.

In a similar vein, a ring \( R \) is called a {\it UU-ring} if \( U(R) = 1 + Nil(R) \) (see, e.g., \cite{12}). As a logical extension of UU-rings, Sheibani and Chen introduced the concept of {\it 2-UU} rings in \cite{10}: a ring \( R \) is classified as 2-UU if the square of each unit can be expressed as the sum of 1 and a nilpotent element. They proved there that a ring \( R \) is strongly 2-nil-clean if, and only if, \( R \) is simultaneously exchange and 2-UU.

Further, as a common expansion of UU-rings and UJ-rings, Ko\c{s}an et al. defined the concept of UNJ rings in \cite{22}. It is well-established that the inclusion \( 1 + Nil(R) + J(R) \subseteq U(R) \) is always valid. A ring \( R \) is classified as a UNJ-ring if the reverse inclusion holds as well, that is, \( U(R) = 1 + Nil(R) + J(R) \). They demonstrated that, if \( R \) is a semi-local ring, then \( R \) is UJ if, and only if, \( R \) is UNJ, and this is tantamount to the existence of some \( n \in \mathbb{N} \) such that \( R/J(R) \cong \mathbb{Z}_{2^n} \). Furthermore, they proved there that any UNJ-ring is Dedekind-finite, and that a UNJ-ring is semi-regular if, and only if, it is an exchange ring, which, in turn, is equivalent to being a clean ring.

Building on the aforementioned concepts, we introduce a broader class of rings, termed {\it 2-UNJ rings}: a ring \( R \) is called 2-UNJ if, for every \( u \in U(R) \), we have \( u^2 = 1 + q + j \), where \( q \in Nil(R) \) and \( j \in J(R) \). It is pretty evident that each UJ, UU and UNJ ring is a 2-UNJ ring, although the converse is manifestly {\it not} generally true. Similarly, while all 2-UU and 2-UJ rings are 2-UNJ, the converse does {\it not} necessarily hold in all generality. Our objective in the present paper is to provide a detailed exploration of these 2-UNJ rings, focusing on comparing their fundamental properties with those of 2-UU and 2-UJ rings. In addition, we aim to identify new and fascinating characteristics of 2-UNJ rings that have {\it not} been thoroughly examined or widely discussed in the existing body of literature.

Let us now briefly review some classical concepts that was mentioned above and which will play a crucial role in our upcoming discussions below. A ring \( R \) is called {\it boolean} if every element of \( R \) is an idempotent. More generally, a ring \( R \) is termed {\it regular} (respectively, {\it unit-regular}) in the sense of von Neumann, provided, for each \( a \in R \), there exists \( x \in R \) (respectively, \( x \in U(R) \)) such that \( axa = a \). Additionally, a ring \( R \) is said to be {\it strongly regular} if, for any \( a \in R \), we have \( a \in a^2R \).

Recall likewise that a ring \( R \) is called {\it exchange} if, for every \( a \in R \), there exists an idempotent \( e \in R \) such that \( e \in aR \) and \( 1 - e \in (1 - a)R \). Besides, a ring \( R \) is defined as {\it clean} if each element of \( R \) can be written as the sum of an idempotent and a unit (cf. \cite{8}). It is principally known that every clean ring is exchange; however, the converse is {\it not} generally true, although it holds in the abelian case (see, for instance, \cite[Proposition 1.8]{8}).

Similarly, a ring \( R \) is termed {\it semi-regular} if the quotient ring \( R/J(R) \) is regular and each idempotent in \( R/J(R) \) lifts to an idempotent in \( R \). It is well known that every semi-regular ring is exchange, though the converse does {\it not} always hold (see, e.g., \cite{8}).

Furthermore, mimicking the work of Chen (\cite{16}), an element of a ring is said to be {\it \( J \)-clean} if it can be expressed as the sum of an idempotent and an element from the Jacobson radical, and a ring \( R \) is called \( J \)-clean if each element of \( R \) is \( J \)-clean. Equivalently, $R$ is \( J \)-clean exactly when \( R/J(R) \) is boolean and idempotents lift modulo \( J(R) \); in order to keep a record straight, such rings are also referred to as {\it semi-boolean} in \cite{25}. In addition, it was shown in \cite{14} that a ring \( R \) is \( J \)-clean precisely when it is a clean \( UJ \)-ring.

\medskip

Now, we have at our disposal the following diagram, which violates the relationships between the mentioned above classes of rings:

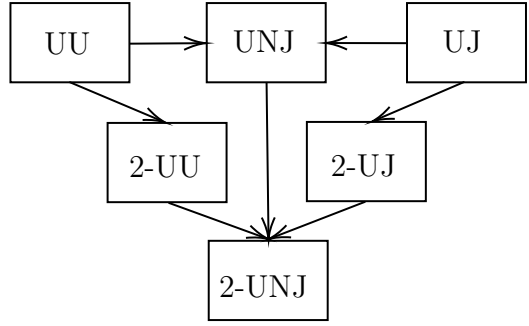
\begin{center}
\tikzset{every picture/.style={line width=0.75pt}} 

\begin{tikzpicture}[x=0.75pt,y=0.75pt,yscale=-1,xscale=1]

\draw   (41,50) -- (101,50) -- (101,90) -- (41,90) -- cycle ;
\draw    (101.17,70.5) -- (137.5,70.26) ;
\draw [shift={(139.5,70.25)}, rotate = 179.63] [color={rgb, 255:red, 0; green, 0; blue, 0 }  ][line width=0.75]    (10.93,-3.29) .. controls (6.95,-1.4) and (3.31,-0.3) .. (0,0) .. controls (3.31,0.3) and (6.95,1.4) .. (10.93,3.29)   ;
\draw    (120.73,150.8) -- (169.29,168.48) ;
\draw [shift={(171.17,169.17)}, rotate = 200.01] [color={rgb, 255:red, 0; green, 0; blue, 0 }  ][line width=0.75]    (10.93,-3.29) .. controls (6.95,-1.4) and (3.31,-0.3) .. (0,0) .. controls (3.31,0.3) and (6.95,1.4) .. (10.93,3.29)   ;
\draw    (240.5,70.25) -- (201.83,70.17) ;
\draw [shift={(199.83,70.17)}, rotate = 0.12] [color={rgb, 255:red, 0; green, 0; blue, 0 }  ][line width=0.75]    (10.93,-3.29) .. controls (6.95,-1.4) and (3.31,-0.3) .. (0,0) .. controls (3.31,0.3) and (6.95,1.4) .. (10.93,3.29)   ;
\draw    (270,89.75) -- (225.34,108.96) ;
\draw [shift={(223.5,109.75)}, rotate = 336.73] [color={rgb, 255:red, 0; green, 0; blue, 0 }  ][line width=0.75]    (10.93,-3.29) .. controls (6.95,-1.4) and (3.31,-0.3) .. (0,0) .. controls (3.31,0.3) and (6.95,1.4) .. (10.93,3.29)   ;
\draw    (220,150.25) -- (173.03,168.44) ;
\draw [shift={(171.17,169.17)}, rotate = 338.83] [color={rgb, 255:red, 0; green, 0; blue, 0 }  ][line width=0.75]    (10.93,-3.29) .. controls (6.95,-1.4) and (3.31,-0.3) .. (0,0) .. controls (3.31,0.3) and (6.95,1.4) .. (10.93,3.29)   ;
\draw    (70.67,89.83) -- (117.16,109.47) ;
\draw [shift={(119,110.25)}, rotate = 202.9] [color={rgb, 255:red, 0; green, 0; blue, 0 }  ][line width=0.75]    (10.93,-3.29) .. controls (6.95,-1.4) and (3.31,-0.3) .. (0,0) .. controls (3.31,0.3) and (6.95,1.4) .. (10.93,3.29)   ;
\draw    (170.17,89.83) -- (171.14,167.17) ;
\draw [shift={(171.17,169.17)}, rotate = 269.28] [color={rgb, 255:red, 0; green, 0; blue, 0 }  ][line width=0.75]    (10.93,-3.29) .. controls (6.95,-1.4) and (3.31,-0.3) .. (0,0) .. controls (3.31,0.3) and (6.95,1.4) .. (10.93,3.29)   ;
\draw   (141,170) -- (201,170) -- (201,210) -- (141,210) -- cycle ;
\draw   (90,110.5) -- (150,110.5) -- (150,150.5) -- (90,150.5) -- cycle ;
\draw   (190.5,110) -- (250.5,110) -- (250.5,150) -- (190.5,150) -- cycle ;
\draw   (241,50) -- (301,50) -- (301,90) -- (241,90) -- cycle ;
\draw   (140,50) -- (200,50) -- (200,90) -- (140,90) -- cycle ;

\draw (223.5,132) node [align=left] {\begin{minipage}[lt]{31.28pt}\setlength\topsep{0pt}
2-UJ
\end{minipage}};
\draw (172.33,192.75) node [align=left] {\begin{minipage}[lt]{38.99pt}\setlength\topsep{0pt}
2-UNJ
\end{minipage}};
\draw (120,132.5) node [align=left] {\begin{minipage}[lt]{28.79pt}\setlength\topsep{0pt}
2-UU
\end{minipage}};
\draw (168,69) node [align=left] {\begin{minipage}[lt]{22.33pt}\setlength\topsep{0pt}
UNJ
\end{minipage}};
\draw (273,69) node [align=left] {\begin{minipage}[lt]{21.08pt}\setlength\topsep{0pt}
UJ
\end{minipage}};
\draw (75,70) node [align=left] {\begin{minipage}[lt]{25.61pt}\setlength\topsep{0pt}
UU
\end{minipage}};

\end{tikzpicture}
\end{center}

\section{Examples and Basic Properties of 2-UNJ Rings}

In the current section, we present illustrative examples and establish fundamental properties of 2-UNJ rings. Through a variety of constructions and counterexamples, we target to clarify the boundaries of the 2-UNJ class and demonstrate how it logically encompasses several previously studied structures such as UJ, UU and UNJ rings. These foundational insights not only enrich our understanding of the algebraic behavior of 2-UNJ rings, but also set the stage for the deeper structural results examined in subsequent sections.

\medskip

We begin here with our main instrument.

\begin{definition}
A ring $R$ is called {\it 2-$UNJ$} if, for each $u\in U(R)$, $u^2=1+j+q$, where $j\in J(R)$ and $q\in Nil(R)$.	
\end{definition}

The next constructions are pivotal for motivating our further writing.

\begin{example}\label{con xemple}
(1) Every 2-UJ ring is a 2-UNJ ring. However, the converse does {\it not} necessarily hold. Indeed, consider the ring $A = \mathbb{F}_3\langle x,y : x^2=0 \rangle$. Consulting with \cite[Lemma 3.18]{ns}, we have:
\[ U(A) = \{\mu + \mathbb{F}_3x + xAx : \mu \in U(F_3)\}, \]
\[ Nil(A) = \{\mathbb{F}_3x + xAx\}. \]
It can easily be checked that $J(A) = (0)$. Now, we prove that $A$ is a 2-UNJ ring, but {\it not} 2-UJ. For this purpose, set $u := \mu + \alpha x + x\beta x \in U(A)$, where $\alpha \in \mathbb{F}_3$ and $\beta \in A$. Thus, one calculates that
\[ u^2 = \mu^2 + 2\mu(\alpha x + x\beta x) \in 1 + Nil(A), \]
so $A$ is a 2-UU ring and hence 2-UNJ. However, it is quite clear that $A$ is {\it not} 2-UJ, as pursued, because $1 + x \in U(A)$, but $(1 + x)^2 = 1 - x + x^2 \notin 1 + J(A)$.

\medskip
	
(2) Every 2-UU ring is a 2-UNJ ring. However, the converse does {\it not} necessarily hold. Indeed, consider the ring $B = \mathbb{F}_2[[x]]$. Exploiting \cite[Example 2.1]{22}, $B$ is a UNJ ring and hence 2-UNJ. But, for $1 + x \in U(B)$, we compute $(1 + x)^2 = 1 + x^2 \notin 1 + Nil(B)$ and, therefore, $B$ is {\it not} a 2-UU ring, as asked for.
\end{example}

Particularly, we derive:

\begin{example}\label{3.32}
The ring $\mathbb{Z}_3$ is 2-$UNJ$, but is {\it not} $UNJ$.	
\end{example}

We now proceed by proving some elementary but useful properties.

\begin{proposition}\label{3.1}
Every finite direct product of 2-UNJ rings is again a 2-UNJ ring.
\end{proposition}

\begin{proof}
This is rather obvious, so we omit the details.
\end{proof}

\begin{proposition}\label{3.2}
Let \(R\) be a 2-UNJ ring. If \(T\) is a factor ring of \(R\) such that all units of \(T\) lift to units of \(R\), then \(T\) is 2-UNJ too.
\end{proposition}

\begin{proof}
Suppose that \(f: R \rightarrow T\) is a ring epimorphism. Choosing \(v \in U(T)\), there exists \(u \in U(R)\) such that \(v = f(u)\) and \(u^2 = 1 + j+q \in 1 + J(R)+Nil(R)\). So, it must be that \[v^2 = (f(u))^2 = f(u^2) = f(1 + j+q) = 1 + f(j)+f(q) \in 1 + J(T)+Nil(T),\] as required.
\end{proof}

The following criterion is helpful as well.

\begin{proposition}\label{3.3}
A division ring \(R\) is 2-UNJ if, and only if, either \(R \cong \mathbb{Z}_2\) or \(R \cong \mathbb{Z}_3\).
\end{proposition}

\begin{proof}
Since \(R\) is a division ring, one has that $J(R)=Nil(R)=(0)$. Consequently, for any non-zero element $a \in R$, where $R  \setminus \{0\} = U(R)$, we obtain $a^2 = 1$, thus leading to $a^3 = a$. The application of the classical Jacobson's theorem yields that $R$ is necessarily commutative.

Consider next the polynomial $f(x) = 1 - x^2 \in R[x]$. As $R$ is a field, $f(x)$ can have at most two roots in the multiplicative group $R^{\ast}$. Denoting by $A$ the collection of all roots of $f$ in $R^{\ast}$, and since $R$ is both a 2-UNJ ring and a field, we observe that, for any non-zero element $a \in R$, the equality $a^2 = 1$ holds. This establishes that $R^* = A$ and, consequently, $|R^{\ast}| = |A| \leq 2$. So, $R$ must be isomorphic to one of the finite field $\mathbb{Z}_2$ or $\mathbb{Z}_3$ either. The reverse direction of this statement is very straightforward to verify, so we drop off the arguments.
\end{proof}

\begin{remark}
The condition "all units of $T$ lift to units of $R$" in Proposition \ref{3.2} is necessary and cannot be ignored. In fact, to substantiate this, one sees that the ring $\mathbb{Z}_7$ is a factor-ring of the 2-$UNJ$ ring $\mathbb{Z}$. But, $\mathbb{Z}_7$ is {\it not} 2-$UNJ$ owing to Proposition \ref{3.3}. Note also that {\it not} all of units of $\mathbb{Z}_7$ can lift to units of $\mathbb{Z}$.
\end{remark}

We continue with a series of technicalities as follows.

\begin{proposition}\label{3.5}
(i) If \( R \) is a 2-UNJ ring, then, for any ideal \( I \) of \( R \) contained in \( J(R) \), \( R/I \) is also a 2-UNJ ring.

(ii) Let \( I \) be a nil-ideal in \( R \). If \( R/I \) is a 2-UNJ-ring, then \( R \) is a 2-UNJ ring.

(iii) If \( R \) is a 2-UNJ ring, then \( R/J(R) \) is a 2-UU ring.

(iv) If \( R \) is a 2-primal 2-UNJ ring, then \( R \) is a 2-UJ ring.
\end{proposition}

\begin{proof}
(i) This is a direct consequence of Proposition \ref{3.2}.

(ii) Assume that $u \in U(R)$ and $\overline{R}=R/I$. Thus, $\bar{u} \in \overline{R}$, so there exist $\bar{j} \in J(\overline{R})$ and $\bar{q} \in \text{Nil}(\overline{R})$ such that $\bar{u} = \bar{1} + \bar{q} + \bar{j}$.

Note that, as $I$ is a nil-ideal, we must have $I \subseteq J(R)$, which forces $J(\overline{R}) = J(R)/I$. Therefore, $j \in J(R)$.

Moreover, since $\bar{q} \in \text{Nil}(\overline{R})$, there exists $n \in \mathbb{N}$ such that $q^n \in I \subseteq \text{Nil}(R)$, and thus $q \in \text{Nil}(R)$.

Consequently, $u - (1 + q + j) \in I$, so there exists $p \in I$ such that $u = 1 + (q + p) + j$.

It just remains to establish that $q + p \in \text{Nil}(R)$. To this goal, since $q^n \in I$, we deduce $(q + p)^n = q^n + f(p,q)$, where $f(p,q) \in I$. This gives that $(q + p)^n \in I \subseteq \text{Nil}(R)$, and hence $q + p \in \text{Nil}(R)$, completing the proof.

(iii) Choose $u+J(R)\in U(R/J(R)$. Thus, $u\in U(R)$ and so $u^2=1+q+j$, where $q\in Nil(R)$ and $j\in J(R)$. Therefore, $$(u+J(R))^2=(1+q+j)+J(R)=(1+J(R))+(q+J(R)),$$ where $q+J(R)$ is a nilpotent in $R/J(R)$.

(iv) Suppose that \( R \) is a 2-primal ring. Then, the set of nilpotent elements \( Nil(R) \) coincides with the prime radical of \( R \) and, consequently, we find that \( Nil(R) \subseteq J(R) \). Now, choosing $u\in U(R)$, so by hypothesis $u^2=1+j+q$, where $j\in J(R)$ and $q\in Nil(R)\subseteq J(R)$. Finally, $u^2= 1+j'\in 1+J(R)$, as required.
\end{proof}

\begin{proposition}\label{3.8}
For any ring \( R \neq 0 \) and any integer \( n \geq 2 \), the ring \( M_n(R) \) is not a 2-UNJ ring.
\end{proposition}

\begin{proof}
Suppose the contrary that \( M_n(R) \) is a 2-UNJ ring. So, viewing Proposition \ref{3.5}(iii), we detect that
\( M_n(R)/J(M_n(R)) \cong M_n(R/J(R)) \) is a 2-UU ring. However, this leads to a contradiction with \cite[Example 2.1]{10}.
\end{proof}

Let us remember now that the set $\{e_{ij} : 1 \le i, j \le n\}$ of non-zero elements of $R$ is said to be a {\it system of $n^2$ matrix units}, provided that $e_{ij}e_{st} = \delta_{js}e_{it}$, where $\delta_{jj} = 1$ and $\delta_{js} = 0$ for $j \neq s$. In this case, $e := \sum_{i=1}^{n} e_{ii}$ is an idempotent of $R$ and the isomorphism $eRe \cong M_n(S)$ is valid, where $$S = \{r \in eRe : re_{ij} = e_{ij}r,~~\textrm{for all}~~ i, j = 1, 2, . . . , n\}.$$
Recall also that a ring $R$ is said to be {\it Dedekind-finite} if $ab=1$ assures that $ba=1$ for any $a,b\in R$. In other words, all one-sided inverses in such a ring are mandatory to be two-sided.

\medskip

We now come to the following statement.

\begin{proposition}\label{3.29}
Every 2-UU ring is Dedekind-finite.
\end{proposition}

\begin{proof}
If we assume the contrapositive that $R$ is {\it not} a Dedekind-finite ring, then there exist elements $a, b \in R$ such that $ab = 1$ but $ba \neq 1$. Assuming $e_{ij} = a^i(1-ba)b^j$ and $e =\sum_{i=1}^{n}e_{ii}$, there exists a non-zero ring $S$ such that $eRe \cong M_n(S)$. However, according to \cite[Proposition 2.1]{10}(3), $eRe$ is a 2-$UU$ ring, whence $M_n(S)$ must also be a 2-$UU$ ring, thus contradicting \cite[Example 2.1]{10}. Finally, $R$ is a Dedekind-finite ring, as stated.
\end{proof}

We slightly improve the last assertion to the following one.

\begin{proposition}\label{3.28}
Every 2-UNJ ring is Dedekind-finite.
\end{proposition}

\begin{proof}
Suppose on the opposite claim that \( R \) is {\it not} Dedekind-finite, and choose \( a, b \in R \) such that \( ab = 1 \) and \( e = ba \neq 1 \). Then, \( 0 \neq e \notin J(R) \) and \( 1 - e \notin J(R) \). Hence, \( R/J(R) \) is not Dedekind-finite. Furthermore, since \( R \) is 2-UNJ, we have that \( R/J(R) \) is 2-UU thanks to Proposition \ref{3.5}(iii), leading to a contradiction with Proposition \ref{3.29}, as expected.
\end{proof}

As a valuable consequence, we extract:

\begin{corollary}\label{3.26} The following three items are fulfilled:

(i) A ring \( R \) is both local and 2-UNJ if, and only if, either \( R/J(R) \cong \mathbb{Z}_2 \) or \(R/J(R) \cong  \mathbb{Z}_3 \).

(ii) A semi-simple ring $R$ is 2-UNJ if, and only if, $R \cong \bigoplus_{i=1}^n R_i$, where $R_i\cong \mathbb{Z}_2 \text{ or } R_i\cong \mathbb{Z}_3$ for every index $i$.

(iii) A ring \( R \) is both semi-local and 2-UNJ if, and only if, $R/J(R) \cong \bigoplus_{i=1}^m R_i$, where $R_i \cong \mathbb{Z}_2 \text{ or } R_i\cong \mathbb{Z}_3$ for every index $i$.
\end{corollary}

\begin{proof}
(i) This follows automatically from a combination of Example \ref{3.3} and Proposition \ref{3.5}(i). Now, suppose that either $R/J(R) \cong \mathbb{Z}_2$ or $R/J(R) \cong \mathbb{Z}_3$. Then, $R/J(R)$ is a 2-UJ ring and, working with \cite[Lemma 2.5]{1}, one inspects that $R$ is a 2-UJ ring, insuring that $R$ is a 2-UNJ ring.

(ii) If $R$ is a semi-simple ring, then applying the classical Wedderburn-Artin theorem, we have an isomorphism $$R \cong \bigoplus_{i=1}^n M_{n_i}(D_i),$$ where each $D_i$ is a division ring. Since $R$ is 2-UNJ, Proposition \ref{3.1} ensures that each $M_{n_i}(D_i)$ must also be a 2-UNJ ring. Furthermore, employing Proposition \ref{3.8}, it follows that $n_i=1$ for all $i$, so every $D_i$ itself is 2-UNJ too. Looking at Proposition \ref{3.3}, we deduce that any $D_i$ is isomorphic to either $\mathbb{Z}_2$ or $\mathbb{Z}_3$, as requested.

Conversely, if $R$ is a finite direct sum of copies of either $\mathbb{Z}_2$ and $\mathbb{Z}_3$, then again Proposition \ref{3.3} applies together with Proposition \ref{3.1} to conclude that $R$ is 2-UNJ, as required.

(iii) Necessity is immediate from parts (i) and (ii). Now, to treat the sufficiency, suppose that $R/J(R) \cong \bigoplus_{i=1}^m R_i$, where either $R_i \cong \mathbb{Z}_2$ or $R_i \cong \mathbb{Z}_3$ for every index $i$. Thus, $R/J(R)$ is a 2-UJ ring and, invoking \cite[Lemma 2.5]{1}, $R$ is a 2-UJ ring, guaranteeing that $R$ is a 2-UNJ ring, as needed.
\end{proof}

\begin{lemma}\label{3.12}
Let \(R\) be a 2-UNJ ring. If \(J(R) = (0)\) and every non-zero right ideal of \(R\) contains a non-zero idempotent, then \(R\) is reduced.
\end{lemma}

\begin{proof}
Since $R$ is a 2-UNJ ring with $J(R) = (0)$, we find that $R$ is a 2-UU ring. Suppose on the reciprocity that $R$ is {\it not} reduced. Then, there exists $0 \neq a \in R$ such that $a^2 = 0$. Utilizing \cite[Theorem 2.1]{3}, there exists a non-zero idempotent $e \in R$ with $eRe \cong M_2(T)$ for some non-trivial ring $T$.

On the other hand, using \cite[Corollary 2.23]{djm}, $eRe$ is a 2-UU ring. Therefore, $M_2(T)$ is also a 2-UU ring, which obviously contradicts \cite[Example 2.1]{10}. Hence, $R$ must be reduced, as asserted.
\end{proof}

\begin{lemma} \label{sum two unit}
Let $R$ be a 2-UNJ ring and put $\overline{R} := R/J(R)$. The following two points hold:
	
(i) For any $u, v \in U(R)$, $u^2 + v \neq 1$.
	
(ii) For any $\bar{u}^2, \bar{v} \in U(\overline{R})$, $\bar{u}^2 + \bar{v} \neq \bar{1}$.
\end{lemma}

\begin{proof}
(i) Suppose on the reverse $u^2 + v = 1$. Since $R$ is a 2-$UNJ$ ring, there exists $q \in Nil(R)$ and $j\in J(R)$ such that $u^2 = 1 + q+j$. Then, $1 = u^2 + v = 1 + q +j+ v$, which means $q\in U(R)$, and hence $q=0$. Thus, $j\in U(R)$, which is impossible.
	
(ii) Suppose on the reverse $\overline{u}^2 + \overline{v} = \overline{1}$. We, without loss of generality, may assume $u, v \in U(R)$, whence $u^2 + v - 1 \in J(R)$. On the other side, as $R$ is a 2-$UNJ$ ring, there exist $q \in Nil(R)$ and $j\in J(R)$ such that $u^2 = 1 + q+j$, giving that $q +j+ v \in J(R)$. Therefore, $q \in U(R) + J(R) \subseteq U(R)$, again a false.
\end{proof}

\begin{lemma}\label{exe}
Let $R$ be a potent 2-UNJ ring and $\overline{R} = R/J(R)$. The next two issues are true:
	
(i) For any idempotent $\bar{e} \in \overline{R}$ and units $\bar{u}, \bar{v} \in U(\bar{e}\overline{R}\bar{e})$, we have $\bar{u}^2 + \bar{v} \neq \bar{e}$.
	
(ii) For any $n>1$, there are no idempotents $\bar{e} \in \overline{R}$ for which $\bar{e}\overline{R}\bar{e}$ is isomorphic to $M_n(S)$ for any (non-zero) ring $S$.
\end{lemma}

\begin{proof}
(i) Since \( R \) is 2-UNJ, it must be that \( \overline{R} \) is 2-UU in accordance with Proposition \ref{3.5}(iii), and so \( \bar{e}\overline{R}\bar{e} \) is 2-UU referring to \cite[Proposition 2.1]{10}(3). Suppose that \( \bar{u}^2 + \bar{v} = \bar{e} \). Since \( \bar{e}\overline{R}\bar{e} \) is 2-UU, there exists \( \bar{q} \in Nil(\bar{e}\overline{R}\bar{e}) \) such that \( \bar{u}^2 = \bar{e} + \bar{q} \). So, \( \bar{e} = \bar{e} + \bar{q} + \bar{v} \), and hence \( \bar{q} \in U(\bar{e}\overline{R}\bar{e}) \), an absurd.
	
(ii) In contrast with the claimed, suppose there exists \( \bar{e} \in \overline{R} \) such that \(\bar{e}\overline{R}\bar{e} \cong M_n(S)\) for some non-zero ring \( S \). Since \( R \) is 2-UNJ, it follows at once that \( \overline{R} \) is 2-UU, whence \( \bar{e}\overline{R}\bar{e} \) is also 2-UU. Consequently, \( M_n(S) \) is 2-UU, which is an impossibility.
\end{proof}

\section{Main Results}

In this key section, we systematically develop the core theoretical framework of 2-UNJ rings: in fact, we establish several deep characterizations, reveal their intrinsic connections with well-known classes of rings such as potent, tripotent, regular and exchange rings, as well as, moreover, we delineate precise conditions under which a ring exhibits the 2-UNJ property. The results obtained here not only consolidate the foundational role of 2-UNJ rings within contemporary ring theory, but also open new avenues for further algebraic exploration in different aspects. The proofs are constructed with careful attention to both generality and structural insight, ensuring a comprehensive understanding of this emerging class.

We remember that a ring $R$ is {\it semi-potent} if every one-sided ideal not contained in $J(R)$ contains a non-zero idempotent.

\medskip

Our first chief results states thus.

\begin{theorem}\label{semipotent}
Let $R$ be a semi-potent ring. The following five claims are equivalent:
	
(i) $R$ is a $2$-UNJ ring.

(ii) $R/J(R)$ is a $2$-UNJ ring.
	
(iii) $R/J(R)$ is a tripotent ring.
	
(iv) $R$ is a $2$-UJ ring.
	
(v) $R/J(R)$ is a $2$-UU ring.
\end{theorem}

\begin{proof}
The implications (i) $\Rightarrow$ (ii), (iv) $\Rightarrow$ (i), (iv) $\Rightarrow$ (v) and (v) $\Rightarrow$ (ii) are quite obvious, so we remove the details.
	
(ii) $\Rightarrow$ (iii). Having in mind Lemma \ref{3.12}, we can conclude that $R/J(R)$ is reduced and hence abelian. Now, suppose there exists $a \in R/J(R)$ such that $a-a^3 \neq 0$ in $R/J(R)$. Since $R/J(R)$ is semi-potent, there exists $e=e^2 \in R/J(R)$ with $e \in (a-a^3)R/J(R)$. We thus get $e=(a-a^3)b$ for some $b \in R/J(R)$.
	
But, as $R/J(R)$ is abelian, we have:
\begin{equation*}
		e = ea(1-a^2)b = e(1-a^2)ab.
\end{equation*}
Thus, with Proposition \ref{3.28} at hand, we infer $ea, e(1-a^2) \in U(eR/J(R)e)$. 

On the other hand, we know that
\begin{equation*}
		(ea)^2 + e(1-a^2) = e. \qquad (*)
\end{equation*}
So, equation $\left( \ast \right)$ contradicts Lemma \ref{exe}, whence $R/J(R)$ is tripotent.
	
(iii) $\Rightarrow$ (iv). Let $u \in U(R)$. Then, in virtue of (ii), we have $u-u^3 \in J(R)$. Moreover, since $J(R)$ is an ideal and $u$ is a unit, we can write $1-u^2 \in J(R)$. Therefore, $R$ is a 2-UJ ring.
\end{proof}

Recall once again that a ring $R$ is called {\it $\pi$-regular} if, for each $a\in R$, the relation $a^n\in a^nRa^n$ holds for some integer $n\ge1$. Thus, regular rings are always $\pi$-regular. Also, a ring $R$ is said to be {\it strongly $\pi$-regular}, provided that, for any $a\in R$, there exists $n\ge 1$ such that $a^n\in a^{n+1}R$.

\medskip

Our next main result is stated as follows.

\begin{theorem}\label{3.13}
Let \(R\) be a ring. Then, the following three assertions are equivalent:
\begin{enumerate}
\item
\(R\) is a regular 2-UNJ ring.
\item
\(R\) is a \(\pi\)-regular reduced 2-UNJ ring.
\item
\(R\) is a tripotent ring.
\end{enumerate}
\end{theorem}

\begin{proof}
(i) \(\Rightarrow\) (ii). Since \(R\) is regular, it must be that \(J(R) = (0)\), and besides every non-zero right ideal contains a non-zero idempotent. Thus, in virtue of Lemma \ref{3.12}, \(R\) is reduced. Likewise, every regular ring is \(\pi\)-regular.

(ii) \(\Rightarrow\) (iii). Notice that reduced rings are always abelian, so \(R\) is strongly \(\pi\)-regular and also $J(R)=Nil(R)=(0)$ by virtue of \cite[Lemma 5]{4}.

Choose \(x \in R\). Bearing in mind \cite [Proposition 2.5]{5}, there are an idempotent \(e \in R\) and a unit \(u \in R\) such that \(x = e + u\) and \(ex = xe \in Nil(R) = (0)\). So, we derive \[x = x - xe = x(1-e) = u(1-e) = (1-e)u.\] But, since \(R\) is a 2-$UNJ$ ring, \(u^2 = 1\). It now follows that \(x^2 = (1-e)\). Hence, \(x = x(1-e) = x.x^2 = x^3\).

(iii) \(\Rightarrow\) (i). It is not too hard to inspect that \(R\) is regular. Choosing \(u \in U(R)\), we obtain \(u^3 = u\), i.e., \(u^2 = 1\), and so\(R\) is a 2-$UNJ$ ring, as asserted.
\end{proof}

Now, we have everything in hand to list the following consequence.

\begin{corollary}\label{3.14}
The following four statements are equivalent for a ring \(R\):
\begin{enumerate}
\item
\(R\) is a regular 2-UNJ ring.
\item
 \(R\) is a strongly regular 2-UNJ ring.
\item
\(R\) is a unit-regular 2-UNJ ring.
\item
\(R\) is a tripotent ring.
\end{enumerate}
\end{corollary}

Imitating Ko\c{s}an et al. (cf. \cite {17}), a ring $R$ is called {\it semi-tripotent} if, for each $a\in R$, $a=e+j$, where $e^3=e$ and $j\in J(R)$ (or, equivalently, $R/J(R)$ satisfies the identity $x^3=x$ and all idempotents that lift modulo $J(R)$).

\medskip

As three important consequences, we record the following corollaries:

\begin{corollary}\label{3.16}
Let \(R\) be a ring. Then, the following three conditions are equivalent:
\begin{enumerate}
\item
\(R\) is a semi-regular 2-UNJ ring.
\item
\(R\) is an exchange 2-UNJ ring.
\item
\(R\) is a semi-tripotent ring.
\end{enumerate}
\end{corollary}

\begin{corollary}\label{3.17}
Let \( R \) be a 2-UNJ ring. Then, the following three items are equivalent:
\begin{enumerate}
\item
\( R \) is a semi-regular ring.
\item
\( R \) is an exchange ring.
\item
\( R \) is a clean ring.
\end{enumerate}
\end{corollary}

\begin{corollary}\label{3.18}
Let \( R \) be a ring. Then, the following two issues are equivalent:
\begin{enumerate}
\item
\( R \) is an exchange 2-UNJ ring and \(J(R)\) is nil.
\item
\( R \) is a strongly $2$-nil-clean ring.
\end{enumerate}
\end{corollary}

\begin{proof}
(i) \(\Rightarrow\) (ii). Since \( R \) is 2-$UNJ$, we get $R/J(R)$ is 2-$UU$ and hence $R$ is 2-$UU$ exploiting \cite[Lemma 2.1]{10}. Therefore, \( R \) is exchange 2-$UU$ ring, whence it is strongly \( 2 \)-nil-clean as \cite[Theorem 4.1]{10} claimed.

(ii) \(\Rightarrow\) (i). This is pretty straightforward, so we are dropping any argumentation.
\end{proof}

\section{Some Extensions of 2-UNJ Rings}

This section is devoted to exploring various natural extensions and generalizations of the 2-UNJ property across broader algebraic contexts. We examine how the 2-UNJ condition behaves under constructions such as trivial extensions, formal matrix rings, skew polynomial rings and, respectively, power series rings. These investigations not only reveal the robustness and versatility of the 2-UNJ structure, but also highlight its deep compatibility with classical and modern ring-theoretic frameworks. Our results pave the way for further applications and demonstrate the potential of 2-UNJ rings to serve as a unifying theme in the study of the contemporary generalized algebraic systems.

\medskip

Let $R$ be a ring, and $M$ a bi-module over $R$. The trivial extension of $R$ and $M$ is defined as
\[ T(R, M) = \{(r, m) : r \in R \text{ and } m \in M\}, \]
with addition defined component-wise and multiplication defined by
\[ (r, m)(s, n) = (rs, rn + ms). \]
Note that the trivial extension $T(R, M)$ is isomorphic to the subring
\[ \left\{ \begin{pmatrix} r & m \\ 0 & r \end{pmatrix} : r \in R \text{ and } m \in M \right\} \]
consisting of the formal $2 \times 2$ matrix ring $\begin{pmatrix} R & M \\ 0 & R \end{pmatrix}$, and likewise $T(R, R) \cong R[x]/\left\langle x^2 \right\rangle$. We also notice that the set of units of the trivial extension $T(R, M)$ is
\[ U(T(R, M)) = T(U(R), M). \]

We now continue by establishing the following series of affirmations.

\begin{lemma}
Let \( R \) be a ring, and \( M \) an \( R \)-bimodule. Then,
	\[
	Nil(T(R, M)) = T(Nil(R), M).
	\]
\end{lemma}

\begin{proof}
The equality follows from a routine verification, and thus we omit the full details, leaving them to the interested reader for a direct check.
\end{proof}

Let $R$ be a ring and let $\alpha : R \to R$ be a ring endomorphism. Standardly, $R[[x; \alpha]]$ designs the {\it ring of skew formal power series} over $R$; that is, the ring of all formal power series in $x$ with coefficients from $R$ with multiplication defined by $xr = \alpha(r)x$ for all $r \in R$. Particularly, $R[[x]] = R[[x; 1_R]]$ designates the ordinary {\it ring of formal power series} over $R$.

\medskip

We move forward with establishment of the following.

\begin{proposition} \label{cor five}
Let $R$, $S$ be rings, $N$ an $(R,S)$-bimodule, and $M$ an $R$-bimodule. Then, the following hold:
	
(i) The trivial extension $T(R,M)$ is 2-UNJ if, and only if, $R$ is 2-UNJ.
	
(ii) The formal triangular matrix ring $\begin{pmatrix} R & N \\ 0 & S \end{pmatrix}$ is 2-UNJ if, and only if, both $R$ and $S$ are 2-UNJ.
	
(iii) For any $n \geq 1$, the triangular matrix ring $T_n(R)$ is 2-UNJ if, and only if, $R$ is 2-UNJ.
	
(iv) The power series ring $R[[x; \alpha]]$ is 2-UNJ if, and only if, $R$ is 2-UNJ.

(v) The power series ring $R[[x]]$ is 2-UNJ if, and only if, $R$ is 2-UNJ.
\end{proposition}

\begin{proof}
It suffices to prove only case (1), as the remaining cases follow by arguments very analogous to that of (1). 

To that end, set $A:=T(R, M)$ and consider $I:=T(0, M)$. It is not so difficult to verify that $I\subseteq J(A)$ such that $A/I \cong R$. If $A$ is 2-UNJ, then $R$ is too 2-UNJ in view of Proposition \ref{3.5}. 

Conversely, assume that \( R \) is a 2-UNJ ring and let \( (u, m) \in U(T(R, M)) \). Thus, \( u \in U(R) \). next, since \( R \) is a 2-UNJ ring, we have \( u^2-1 \in Nil(R)+J(R) \). It also follows by direct calculations that
\[
(u, m)^2-(1, 0) = (u^2-1, \ast) \in T(Nil(R), M)+ T(J(R), M).
\]
Hence, one obtains that $$(u, m)^2-(1, 0)\in Nil(T(R, M))+J(T(R, M)),$$ and so \( T(R, M) \) is a 2-UNJ ring, as required.
\end{proof}

Suppose $R$ is a ring and $M$ is a bi-module over $R$. Putting $$DT(R,M) := \{ (a, m, b, n) | a, b \in R, m, n \in M \}$$ with addition defined component-wise and multiplication defined by $$(a_1, m_1, b_1, n_1)(a_2, m_2, b_2, n_2) = (a_1a_2, a_1m_2 + m_1a_2, a_1b_2 + b_1a_2, a_1n_2 + m_1b_2 + b_1m_2 +n_1a_2),$$ we then observe that $DT(R,M)$ is a ring which is isomorphic to $T(T(R, M), T(R, M))$. Moreover, we may write $$DT(R, M) =
\left\{\begin{pmatrix}
	a &m &b &n\\
	0 &a &0 &b\\
	0 &0 &a &m\\
	0 &0 &0 &a
\end{pmatrix} |  a,b \in R, m,n \in M\right\}.$$ We now exhibit the following isomorphism of rings: the map $$\dfrac{R[x, y]}{\langle x^2, y^2\rangle} \rightarrow DT(R, R)$$, defined by $$a + bx + cy + dxy \mapsto
\begin{pmatrix}
	a &b &c &d\\
	0 &a &0 &c\\
	0 &0 &a &b\\
	0 &0 &0 &a
\end{pmatrix},$$ can easily be inspected that it gives a ring isomorphism.

\medskip

We now arrive at the following.

\begin{corollary}
Let $R$ be a ring, and $M$ a bi-module over $R$. Then, the following four statements are equivalent:
\begin{enumerate}
\item
$R$ is a 2-UNJ ring.
\item
$DT(R, M)$ is a 2-UNJ ring.
\item
$DT(R, R)$ is a 2-UNJ ring.
\item
$R[x, y]/\langle x^2, y^2\rangle$ is a 2-UNJ ring.
\end{enumerate}
\end{corollary}

Further, let $\alpha$ be an endomorphism of $R$ and suppose $n$ is a positive integer. It was defined by Nasr-Isfahani in \cite{nasr} the concept of {\it skew triangular matrix ring} like this:

$$T_{n}(R,\alpha )=\left\{ \left. \begin{pmatrix}
	a_{0} & a_{1} & a_{2} & \cdots & a_{n-1} \\
	0 & a_{0} & a_{1} & \cdots & a_{n-2} \\
	0 & 0 & a_{0} & \cdots & a_{n-3} \\
	\ddots & \ddots & \ddots & \vdots & \ddots \\
	0 & 0 & 0 & \cdots & a_{0}
\end{pmatrix} \right| a_{i}\in R \right\}$$

\medskip

\noindent with addition point-wise and multiplication given by:

\medskip

\begin{align*}
	&\begin{pmatrix}
		a_{0} & a_{1} & a_{2} & \cdots & a_{n-1} \\
		0 & a_{0} & a_{1} & \cdots & a_{n-2} \\
		0 & 0 & a_{0} & \cdots & a_{n-3} \\
		\ddots & \ddots & \ddots & \vdots & \ddots \\
		0 & 0 & 0 & \cdots & a_{0}
	\end{pmatrix}\begin{pmatrix}
		b_{0} & b_{1} & b_{2} & \cdots & b_{n-1} \\
		0 & b_{0} & b_{1} & \cdots & b_{n-2} \\
		0 & 0 & b_{0} & \cdots & b_{n-3} \\
		\ddots & \ddots & \ddots & \vdots & \ddots \\
		0 & 0 & 0 & \cdots & b_{0}
	\end{pmatrix}  =\\
	& \begin{pmatrix}
		c_{0} & c_{1} & c_{2} & \cdots & c_{n-1} \\
		0 & c_{0} & c_{1} & \cdots & c_{n-2} \\
		0 & 0 & c_{0} & \cdots & c_{n-3} \\
		\ddots & \ddots & \ddots & \vdots & \ddots \\
		0 & 0 & 0 & \cdots & c_{0}
	\end{pmatrix},
\end{align*}
where $$c_{i}=a_{0}\alpha^{0}(b_{i})+a_{1}\alpha^{1}(b_{i-1})+\cdots +a_{i}\alpha^{i}(b_{0}),~~ 1\leq i\leq n-1
.$$

\medskip

We reserve for the elements of $T_{n}(R, \alpha)$ the notation $(a_{0},a_{1},\ldots , a_{n-1})$. In particular, if $\alpha $ is the identity endomorphism, then one apparently verifies that $T_{n}(R,\alpha )$ is a subring of the {\it upper triangular matrix ring} $T_{n}(R)$.

Thereby, we receive the following consequence.

\begin{corollary}
Let $R$ be a ring and $k\geq 1$. Then, the following two points are equivalent:
\begin{enumerate}
\item
$T_{n}(R,\alpha )$ is a 2-UNJ ring.
\item
$R$ is a 2-UNJ ring.
\end{enumerate}
\end{corollary}

Furthermore, a simple manipulation with coefficients shows that there is a ring isomorphism $$\varphi : R[x,\alpha]/(x^n)\rightarrow T_{n}(R,\alpha),$$ given by $$\varphi (a_{0}+a_{1}x+\ldots +a_{n-1}x^{n-1}+\langle x^{n} \rangle )=(a_{0},a_{1},\ldots ,a_{n-1})$$ with $a_{i}\in R$, $0\leq i\leq n-1$. Thus, one checks that the isomorphism $$T_{n}(R,\alpha )\cong R[x,\alpha ]/(x^n)$$ holds, where as usual $(x^n)$ is the ideal generated by $x^{n}$.

\medskip

As two immediate consequences, we detect:

\begin{corollary}
Let $R$ be a ring and $k\geq 1$. Then, the following are tantamount:
\begin{enumerate}
\item
$R$ is a $2$-UNJ ring.
\item
For any $n\geq 2$, the quotient-ring $R[x; \alpha]/(x^n)$ is a $2$-UNJ ring.
\item
For any $n\geq 2$, the quotient-ring $R[[x; \alpha]]/(x^n)$ is a $2$-UNJ ring.
\end{enumerate}
\end{corollary}

\begin{corollary}
Let $R$ be a ring. Then, the following are tantamount:
\begin{enumerate}
\item
$R$ is a $2$-UNJ ring.
\item
For any $n\geq 2$, the quotient-ring $R[x]/(x^n)$ is a $2$-UNJ ring.
\item
For any $n\geq 2$, the quotient-ring $R[[x]]/(x^n)$ is a $2$-UNJ ring.
\end{enumerate}
\end{corollary}

Let $Nil_*(R)$ denote the \textit{prime radical} (also known in the existing literature as the \textit{lower nilradical}) of a ring $R$, that is, the intersection of all prime ideals of $R$. It is well known that $Nil_*(R)$ forms a nil-ideal in $R$.

In this vein, a ring $R$ is said to be {\it $2$-primal} if its prime radical $Nil_*(R)$ coincides with the set of all nilpotent elements of $R$, i.e.,
\[
Nil_*(R) = Nil(R).
\]

Non-trivial examples of $2$-primal rings include \textit{reduced rings} (rings with no non-zero nilpotent elements) and all \textit{commutative rings}, as a plain check shows in both cases that the set of nilpotent elements equals the prime radical.

\medskip

The next lemma is critical for our successful presentation.

\begin{lemma}\cite[Corollary 3.2.]{daoa}\label{cor 2 primal}
Let $R$ be a 2-primal ring. Then:
\[ J(R[x]) = Nil(R)[x] = Nil_*(R)[x] = Nil_*(R[x]). \]
\end{lemma}

Our next principal slightly curious claim is recorded as follows.

\begin{lemma}\label{l4.8}
Let $R$ be a $2$-primal ring. Then, the following four conditions are equivalent:
\begin{enumerate}
\item
$R$ is a $2$-UU ring.
\item
$R[x]$ is a $2$-UNJ ring.
\item
$R[x]$ is a $2$-UJ ring.
\item
$R[x]$ is a $2$-UU ring.
\end{enumerate}
\end{lemma}

\begin{proof}
Implications (iii) $\Rightarrow$ (ii) and (iv) $\Rightarrow$ (ii) are rather obvious, so their verifications are eliminated voluntarily.
	
(i) $\Rightarrow$ (iv). Write $u = \sum_{i=0}^n u_ix^i \in U(R[x])$. Since $R$ is 2-primal, \cite[Theorem 2.5]{Chenpr} discovers that $u_0 \in U(R)$ and, for each $1 \leq i \leq n$, $u_i \in Nil_*(R)$. Thus, since $R$ is 2-UU, we additionally have $1-u_0^2 \in Nil(R)$, and via the ideal property of $Nil_*(R)$, we obtain:
\[
	1-u^2 \in (1-u_0^2) + Nil_*(R)[x]x \subseteq Nil_*(R)[x] = Nil_*(R[x]) \subseteq Nil(R[x]).
\]

(i) $\Rightarrow$ (iii). The proof is similar to the proof of the previous case, and thus ignored.
	
(ii) $\Rightarrow$ (i). Given $u \in U(R) \subseteq U(R[x])$, we discovery that $1-u^2 \in Nil(R[x])+J(R[x])$. Therefore, $1-u^2 \in Nil_*(R) = Nil(R)$.

(i) $\Rightarrow$ (iii). The proof is analogous to the proof of (i) $\Rightarrow$ (iv).
\end{proof}

In what follows, we attempt to extend Lemma \ref{l4.8}. To that extent, we first refine Lemma \ref{cor 2 primal} in a new language: for an endomorphism $\alpha$ of $R$, the ring R is called {\it $\alpha$-compatible} if, for any elements $a$ and $b$ in $R$, the equality $ab = 0$ holds if, and only if, $a\alpha(b) = 0$. This definition is given in full detail \cite{has}. It is worthy of noticing that, in this case, the map $\alpha$ must be injective.

\begin{lemma}
Let $R$ be a 2-primal and $\alpha$-compatible ring. Then, the next equalities are fulfilled:
\[ J(R[x,\alpha]) = \text{Nil}(R)[x,\alpha] = \text{Nil}_*(R)[x,\alpha] = \text{Nil}_*(R[x,\alpha]). \]
\end{lemma}

\begin{proof}
Knowing that $R$ is 2-primal and that \cite[Lemma 2.2]{ccp} is valid, we can write $$\text{Nil}(R)[x,\alpha] = \text{Nil}_*(R)[x,\alpha] = \text{Nil}_*(R[x,\alpha]).$$

On the other hand, we always have $\text{Nil}_*(R[x,\alpha]) \subseteq J(R[x,\alpha])$, so it just suffices to show that $J(R[x,\alpha]) \subseteq \text{Nil}(R)[x,\alpha]$. In fact, writing $f = \sum_{i=0}^{n}a_ix^i \in J(R[x,\alpha])$, we then have $1-fx \in U(R[x,\alpha])$. Therefore, \cite[Lemma 2.2]{ccp} enables us that, for $0\leq i\leq n$, the containment $a_i \in \text{Nil}(R)[x,\alpha]$ is really true, thus proving the desired equalities.
\end{proof}

We are now prepared to prove the following.

\begin{lemma}\label{lemma alpha}
For a $2$-primal $\alpha$-compatible ring $R$, the following assertions are tantamount:
\begin{enumerate}
\item $R$ is a $2$-UU ring.
\item The skew polynomial ring $R[x,\alpha]$ is $2$-UNJ.
\item $R[x,\alpha]$ is a $2$-UJ ring.
\item $R[x,\alpha]$ is a $2$-UU ring.
\end{enumerate}
\end{lemma}

\begin{proof}
The implications (iii) $\Rightarrow$ (ii) and (iv) $\Rightarrow$ (ii) are immediate.

(i) $\Rightarrow$ (iv). Consider $u = \sum_{i=0}^n u_ix^i \in U(R[x])$. By the usage of 2-primality of $R$ and \cite[Lemma 2.2]{ccp}, we know $u_0 \in U(R)$ and $u_i \in Nil_*(R)$ for $1 \leq i \leq n$. Since $R$ is 2-UU, one knows that $1-u_0^2 \in Nil(R)$. Furthermore, using the ideal property of $Nil_*(R)$, we deduce:
\[
1-u^2 \in (1-u_0^2) + Nil_*(R)[x,\alpha]x \subseteq Nil_*(R)[x,\alpha] = Nil_*(R[x]) \subseteq Nil(R[x,\alpha]).
\]

(i) $\Rightarrow$ (iii). It is analogous to the proof of (i) $\Rightarrow$ (iv).

(ii) $\Rightarrow$ (i). For $u \in U(R) \subseteq U(R[x,\alpha])$, one has that $1-u^2 \in Nil(R[x,\alpha])+J(R[x,\alpha])$, leading to $1-u^2 \in Nil_*(R) = Nil(R)$, as needed.

(i) $\Rightarrow$ (iii). It is analogical to the argument for (i) $\Rightarrow$ (iv).
\end{proof}

\medskip

Let $A$, $B$ be two rings, and let $M$, $N$ be an $(A,B)$-bi-module and a $(B,A)$-bi-module, respectively. Also, we consider the two bi-linear maps $\phi :M\otimes_{B}N\rightarrow A$ and $\psi:N\otimes_{A}M\rightarrow B$ that apply to the following properties.
$$Id_{M}\otimes_{B}\psi =\phi \otimes_{A}Id_{M},Id_{N}\otimes_{A}\phi =\psi \otimes_{B}Id_{N}.$$
For $m\in M$ and $n\in N$, we define $mn:=\phi (m\otimes n)$ and $nm:=\psi (n\otimes m)$. Now, the $4$-tuple $R=\begin{pmatrix}
	A & M\\
	N & B
\end{pmatrix}$ becomes to an associative ring with obvious matrix operations, which is termed a {\it Morita context ring}. Likewise, denote the two-sided ideals $Im \phi$ and $Im \psi$ by $MN$ and $NM$, respectively, that are named the {\it trace ideals} of the Morita context ring.

The following necessary and sufficient condition holds true.

\begin{proposition}\label{4.7}
Let $R=\left(\begin{array}{ll}A & M \\ N & B\end{array}\right)$ be a Morita context ring such that $MN$ and $NM$ are nilpotent ideals of $A$ and $B$, respectively. Then, $R$ is a 2-UNJ ring if, and only if, both $A$ and $B$ are 2-UNJ rings.
\end{proposition}

\begin{proof}
Suppose $R$ is 2-UNJ. We prove that $A$ is 2-UNJ too. To show that, suppose \( a \in U(A) \). Next, consider the matrix element
\[
x := \begin{pmatrix} a & 0 \\ 0 & 1 \end{pmatrix} \in U(R).
\]
Since \( R \) is 2-UNJ, we compute
\[
x^2 = \begin{pmatrix} 1 & 0 \\ 0 & 1 \end{pmatrix} +
\begin{pmatrix} r & m_1 \\ n_1 & s \end{pmatrix} +
\begin{pmatrix} j_1 & m_2 \\ n_2 & j_2 \end{pmatrix},
\]
where \( j_1 \in J(A) \), \( j_2 \in J(B) \), and the middle matrix is nilpotent. Therefore, \( r^k + d = 0 \) for some \( d \in MN\). Since $MN$ is nilpotent, we can freely get that \( r \in Nil(A) \), as requested. Similarly, \( B \) is 2-UNJ as well.

Conversely, assume \( A \) and \( B \) are both 2-UNJ rings. Put
\[
x := \begin{pmatrix} a & m \\ n & b \end{pmatrix} \in U(R),
\]
where \( a \in U(A) \), \( b \in U(B) \). Then, we calculate
\[
x^2 = \begin{pmatrix}
	a^2 + mn & am + mb \\
	na + bn & b^2 + nm
\end{pmatrix}.
\]
Since \( A \), \( B \) are 2-UNJ, we derive that $a^2 = 1 + r + j$, where $q \in Nil(A)$ and $j \in J(A)$; also, $b^2 = 1 + s + j_0$, where $s\in Nil(B)$ and $j_0 \in J(B)$.
Furthermore, it must be that
\[
x^2 = \begin{pmatrix} 1 & 0 \\ 0 & 1 \end{pmatrix} +
\begin{pmatrix} r + mn & 0 \\ 0 & s + nm \end{pmatrix} +
\begin{pmatrix} j & am + mb \\ na + bn & j_0 \end{pmatrix},
\]
where the first matrix is nilpotent, and the second one lies in \( J(R) \). Consequently, \( R \) is a 2-UNJ ring, as promised.
\end{proof}

Given a ring $R$ and a central element $s$ of $R$, the $4$-tuple $\begin{pmatrix}
	R & R\\
	R & R
\end{pmatrix}$ becomes a ring with addition component-wise and with multiplication defined by
$$\begin{pmatrix}
	a_{1} & x_{1}\\
	y_{1} & b_{1}
\end{pmatrix}\begin{pmatrix}
	a_{2} & x_{2}\\
	y_{2} & b_{2}
\end{pmatrix}=\begin{pmatrix}
	a_{1}a_{2}+sx_{1}y_{2} & a_{1}x_{2}+x_{1}b_{2} \\
	y_{1}a_{2}+b_{1}y_{2} & sy_{1}x_{2}+b_{1}b_{2}
\end{pmatrix}.$$
This ring is denoted by $K_s(R)$. Moreover, a {\it Morita context}
$\begin{pmatrix}
	A & M\\
	N & B
\end{pmatrix}$ with $A=B=M=N=R$ is said to be a {\it generalized matrix ring} over $R$. It was observed by Krylov in \cite{18} that a ring $S$ is generalized matrix over $R$ if, and only if, $S=K_s(R)$ for some $s\in C(R)$. Here $MN=NM=sR$, so that the equivalencies $MN\subseteq J(A)\Longleftrightarrow s\in J(R)$ and $NM\subseteq J(B)\Longleftrightarrow s\in J(R)$ hold.

\medskip

We can now extract the following consequence, which gives an interesting relationship.

\begin{corollary}\label{4.9}
Let $R$ be a ring and $s\in C(R)\cap Nil(R)$. Then, $K_s(R)$ is a $2$-UNJ ring if, and only if, $R$ is a $2$-UNJ ring.
\end{corollary}

Further, following Tang and Zhou (cf. \cite{19}), for $n\geq 2$ and for $s\in C(R)$, the $n\times n$ {\it formal matrix ring} over $R$ defined via $s$, and denoted by $M_{n}(R;s)$, is the set of all $n\times n$ matrices over $R$ with usual addition of matrices and with multiplication defined below:

\medskip

\noindent For $(a_{ij})$ and $(b_{ij})$ in $M_{n}(R;s)$,
$$(a_{ij})(b_{ij})=(c_{ij}), \quad \text{where} ~~ (c_{ij})=\sum s^{\delta_{ikj}}a_{ik}b_{kj}.$$
Here, $\delta_{ijk}=1+\delta_{ik}-\delta_{ij}-\delta_{jk}$, where $\delta_{jk}$, $\delta_{ij}$, $\delta_{ik}$ are the {\it Kronecker delta symbols}.

\medskip

We now intend to state the following consequence, which provides an interesting relationship.

\begin{corollary}\label{4.10}
Let $R$ be a ring and $s\in C(R)\cap Nil(R)$. Then, $M_{n}(R;s)$ is a 2-UNJ ring if, and only if, $R$ is a 2-UNJ ring.
\end{corollary}

A {\it Morita context} $\begin{pmatrix}
	A & M\\
	N & B
\end{pmatrix}$ is called {\it trivial}, if both the context products are trivial, i.e., $MN=0$ and $NM=0$. We now routinely see that
$$\begin{pmatrix}
	A & M\\
	N & B
\end{pmatrix}\cong T(A\times B, M\oplus N),$$
where
$\begin{pmatrix}
	A & M\\
	N & B
\end{pmatrix}$ is a trivial Morita context referring to \cite{20}.

\medskip

We now manage the following criterion.

\begin{corollary}\label{4.11}
The trivial Morita context
$\begin{pmatrix}
		A & M\\
		N & B
\end{pmatrix}$ is a 2-UNJ ring if, and only if, $A$ and $B$ are both 2-UNJ rings.
\end{corollary}

\begin{proof}
It is elementarily verified that the two isomorphisms
$$\begin{pmatrix}
		A & M\\
		N & B
\end{pmatrix} \cong T(A\times B,M\oplus N) \cong \begin{pmatrix}
		A\times B & M\oplus N\\
		0 & A \times B
\end{pmatrix}$$ hold. Then, the rest of the proof follows combining Propositions \ref{cor five}(i) and \ref{3.1}, as expected.
\end{proof}

\section{Group Rings}

In this section, we investigate the special behavior of the 2-UNJ property within the setting of group rings. By analyzing specific conditions on the underlying groups and employing structural properties of augmentation ideals, we provide further insights into how the 2-UNJ framework interacts with group-theoretic constructions.

Following the traditional terminology, we say that a group $G$ is a {\it $p$-group} if the order of every element of $G$ is a power of the prime number $p$. Moreover, a group $G$ is said to be {\it locally finite} if every finitely generated subgroup is finite.

Suppose now that $G$ is an arbitrary group and $R$ is an arbitrary ring. As usual, the symbol $RG$ stands for the group ring of $G$ over $R$. The homomorphism $\varepsilon :RG\rightarrow R$, defined by $\varepsilon (\displaystyle\sum_{g\in G}a_{g}g)=\displaystyle\sum_{g\in G}a_{g}$, is called the {\it augmentation map} of $RG$ and its kernel, denoted by $\Delta (RG)$, is called the {\it augmentation ideal} of $RG$.

Let $I$ be an ideal of a ring $R$. We define:
\[
\sqrt{I} = \{a \in R : \exists n \in \mathbb{N}, a^n \in I\}.
\]

Considering this notation, we start with the following preliminary claim.

\begin{lemma}
Let R be a ring. Then, we have the following two containments:

(1)  $\text{Nil}(R) + J(R) \subseteq \sqrt{J(R)}$.

(2)  $\sqrt{J(R)} \cap C(R) \subseteq J(R)$.
\end{lemma}

\begin{proof}
(1) Assume $q + j \in Nil(R) + J(R)$, where $q \in Nil(R)$ with $q^n = 0$ and $j \in J(R)$. Thus, it is easy to see that $(q + j)^n \in J(R)$ meaning that $q + j \in \sqrt{J(R)}$.

(2) Assume $a \in \sqrt{J(R)} \cap C(R)$. Thus, there exists $n \in \mathbb{N}$ such that $a^n \in J(R)$. Now, for any $r \in R$, as $ar = ra$, we have $(ra)^n = r^n a^n$. Therefore,
\[
(1 - ra)(1 + (ra) + \cdots + (ra)^{n-1}) = 1 - (ra)^n = 1 - r^n a^n \in U(R),
\]
allowing $1 - ra \in U(R)$. Hence, $a \in J(R)$.
\end{proof}

The next two claims are crucial and so worthy of mentioning.

\begin{lemma}\cite[Proposition 9]{coon}\label{ext le}
Let $R$ be a ring, $G$ a group, and $H$ a subgroup of $G$. 

(1) $J(RG) \cap RH \subseteq J(RH)$.

(2) If $G$ is a locally finite group, then $J(R) = J(RG) \cap R$. In particular, $J(R)G \subseteq J(RG)$.
\end{lemma}

\begin{lemma}
Let $RG$ be a 2-UNJ ring. Then, $R$ is also a 2-UNJ ring.
\end{lemma}

\begin{proof}
Let $u \in U(R) \subseteq U(RG)$. Since $RG$ is a 2-UNJ ring, there exist $q \in Nil(RG)$ and $j \in J(RG)$ such that $u^2 = 1 + q + j$.

Clearly, $\varepsilon(q) \in Nil(R)$ and, since $\varepsilon$ is a surjective homomorphism, we have $\varepsilon(j) \in J(R)$. Therefore, we obtain $u^2 = 1 + \varepsilon(q) + \varepsilon(j)$. Finally, $R$ is a 2-UNJ ring, as claimed.
\end{proof}

In the sequel, we investigate what properties the group $G$ must have when $RG$ is a 2-UNJ ring. Concretely, we succeed to prove the following assertions.

\begin{lemma}
Let $RG$ be a 2-UNJ ring. Then, $G$ is a torsion group.
\end{lemma}

\begin{proof}
Suppose on the reversible hypothesis that there exists $g \in G$ of infinite order. Since $RG$ is a 2-UNJ ring, we can get $$1-g^2 \in \text{Nil}(RG) + J(RG) \subseteq \sqrt{J(RG)}.$$ Therefore, there exists $n \in \mathbb{N}$ such that $$(1-g^2)^n \in J(RG) \cap R\langle g\rangle \subseteq J(R\langle g\rangle).$$

But, as the element $1-g^2$ is central in the ring $R\langle g\rangle$, we have $1-g^2 \in J(R\langle g\rangle)$. Consequently, $1-g+g^2 \in U(R\langle g\rangle)$. Therefore, there exist two integers $n < m$ and elements $a_i$ with $a_n \neq 0 \neq a_m$ such that
\[
(1 - g + g^2)\sum_{i=n}^{m}a_ig^i=1.
\]
This, however, leads to a suspected contradiction, and thus every element $g \in G$ has to be of finite order, as expected.
\end{proof}

With the aid of all of the above, our basic statement sounds thus.

\begin{theorem}\label{2 group ring}
Let $RG$ be a 2-UNJ ring with $2 \in J(R)$. Then, $G$ is a 2-group.
\end{theorem}

\begin{proof}
First, we show that $2 \in J(RG)$. Indeed, since $2 \in J(R)$, we have $3 \in U(R) \subseteq U(RG)$, whence $$2^3 = 3^2 - 1 \in Nil(RG) + J(RG) \subseteq \sqrt{J(RG)}.$$ Therefore, there exists $k \in \mathbb{N}$ such that $2^{3k} \in J(RG)$. As $2$ is central, we conclude $2 \in J(RG)$.

Now, we show that, for any $k \in \mathbb{N}$ and $g \in G$, $\sum_{i=0}^{2k}g^i \in U(RG)$. Assume for a moment that $k=1$. Since $RG$ is 2-UNJ, we have $$1-g^2 \in Nil(RG) + J(RG) \subseteq \sqrt{J(RG)},$$ so there exists $k \in \mathbb{N}$ with $$(1-g^2) \in J(RG) \cap R\langle g\rangle \subseteq J(R\langle g\rangle).$$

But, as $1-g^2$ is central in $R\langle g\rangle$, we may get $1-g^2 \in J(R\langle g\rangle)$. Moreover, since $2 \in J(RG)$, we have $2 \in J(R\langle g\rangle)$, and hence $1+g^2 = (1-g^2) + 2g^2 \in J(R\langle g\rangle)$. Therefore, $$1+g+g^2 = (1+g^2) + g \in J(R\langle g\rangle) + U(R\langle g\rangle) \subseteq U(R\langle g\rangle) \subseteq U(RG),$$ completing this situation.

For $k=2$, from the case $k=1$, we have $1+g^2 \in J(R\langle g\rangle)$, whence:
\[
1+g+g^2+g^3+g^4 = (1+g^2)(1+g) + g^4 \in J(R\langle g\rangle) + U(R\langle g\rangle) \subseteq U(R\langle g\rangle) \subseteq U(RG).
\]

Similarly, we can show for any $k$ that $\sum_{i=0}^{2k}g^i \in U(RG)$ is true.

If now we assume that $g \in G$ has order $p$ not dividing 2, then $p$ must obviously be odd (i.e., $p=2k+1$). However, by what we have proven so far, $\sum_{i=0}^{2k}g^i \in U(RG)$; but, $(1-g)\left(\sum_{i=0}^{2k}g^i\right) = 0$ guarantees $1-g=0$, a visible contradiction. Finally, $G$ must be a 2-group, as formulated.
\end{proof}

Our next major result is the following one.

\begin{theorem}
Let $R$ be a ring with $3 \in J(R)$, and $G$ a $p$-group with $p$ a prime. If $RG$ is a 2-UNJ ring, then either $G$ is a $3$-group or $G$ is a group of exponent $2$.
\end{theorem}

\begin{proof}
We distinguish two cases, namely when $p=2$ and when $p \neq 2$.

\medskip

\textbf{Case 1:} If $p=2$, then $G$ is a 2-group. For any $g \in G$, there exists $n \in \mathbb{N}$ such that $g^{2^n}=1$. Letting $n>1$ be the smallest such integer, since $2^{n-1}$ is even, there exists $m \in \mathbb{N}$ with $2^{n-1}=2m$ such that
\[
1-g^{2^{n-1}}=1-g^{2m} \in Nil(RG)+J(RG) \subseteq \sqrt{J(RG)}.
\]
Thus, there exists $k \in \mathbb{N}$ with $$(1-g^{2^{n-1}})^k \in J(RG) \cap R\langle g\rangle \subseteq J(R\langle g\rangle).$$ But, since $1-g^{2^{n-1}} \in R\langle g\rangle$ is central, we have $1-g^{2^{n-1}} \in J(R\langle g\rangle)$.

Moreover, as $3 \in J(R)$, by analogy with Theorem \ref{2 group ring} we can show that $3 \in J(RG) \cap R\langle g\rangle \subseteq J(R\langle g\rangle)$ and, therefore,
\[
1+2g^{2^{n-1}}=1-g^{2^{n-1}}+3g^{2^{n-1}} \in J(R\langle g\rangle).
\]
Consequently, it must be that 
\[
1+g^{2^{n-1}}=1+2g^{2^{n-1}}-g^{2^{n-1}} \in J(R\langle g\rangle) +U(R\langle g\rangle) \subseteq U(R\langle g\rangle).
\]
However, $$(1-g^{2^{n-1}})(1+g^{2^{n-1}})=1-g^{2^{n}}=0,$$ and since $1+g^{2^{n-1}}\in U(R\langle g\rangle)$, we just have $1-g^{2^{n-1}}=0$, contradicting the minimality of $n$. Thus, $n=1$.

\medskip

\textbf{Case 2:} If $p \neq 2$, then $p$ must be odd. Let $g \in G$ with $g^{p^n}=1$. Suppose also that $p^n=2k+1$. Then, $1-g^{2k} \in Nil(RG)+J(RG)$ and, paralleling to the previous case, we can show $1-g^{2k}\in J(R\langle g\rangle)$. Multiplying this by $-g$ gives that
\[
1-g=-g(1-g^{2k})\in J(R\langle g\rangle).
\]
This forces at once that $\Delta(R\langle g\rangle) \subseteq J(R\langle g\rangle)$ by noticing that $$1-g^n=(1-g)(1+g + \cdots + g^{n-1}) \in J(R\langle g\rangle).$$ However, \cite[Proposition 15(i)]{coon} works to get that $\langle g\rangle$ is a $q$-group with $q \in J(R)$. Since $3\in J(R)$ and two distinct primes cannot be in $J(R)$, we must have $q=3$. So, there exists $m \in \mathbb{N}$ such that $g^{3^m}=1$. And, as $g \in G$ was arbitrary, so $G$ is a 3-group, as pursued.
\end{proof}

We now proceed by proving the following assertion, which is pertained to the opposite consideration of when the group ring is 2-UNJ.

\begin{theorem}\label{T2}
Let $R$ be a $2$-UNJ ring with $2 \in J(R)$, and $G$ a locally finite $2$-group. Then, $RG$ is a $2$-UNJ ring.
\end{theorem}

\begin{proof}
Since $RG/\Delta(RG) \cong R$, we may assume $RG = \Delta(RG) + R$. Let $u \in U(RG)$ with $u = u_0 + f$, where $u_0 \in R$ and $f \in \Delta(RG)$.

First, we show $u_0 \in U(R)$. Since $u \in U(RG)$, there exists $v = v_0 + g \in RG$ with $v_0 \in R$, $g \in \Delta(RG)$ such that $uv = vu = 1$ implying $1 - u_0v_0, 1 - v_0u_0 \in \Delta(RG)$.

Applying now the augmentation map $\varepsilon$, we get $$\varepsilon(1 - u_0v_0) = 0 = \varepsilon(1 - v_0u_0).$$ Since $u_0, v_0 \in R$, this yields $1 - u_0v_0 = 0 = 1 - v_0u_0$, proving $u_0 \in U(R)$.

But, as $R$ is 2-UNJ, we have $u_0^2 = 1 + q + j$, where $q \in Nil(R)$ and $j \in J(R)$. Since $\Delta(RG)$ is an ideal, $u^2 = u_0^2 + f'$ with $f' \in \Delta(RG)$.

Furthermore, Lemma \ref{ext le} employs to get that $J(R) \subseteq J(RG)$ and, by \cite[Lemma 2]{zhouc}, $\Delta(RG) \subseteq J(RG)$. Therefore, we receive
\[
u^2 = 1 + n + j + f' \in 1 + Nil(RG) + J(RG) + J(RG) = 1 + Nil(RG) + J(RG).
\]
So, $RG$ is a 2-UNJ ring, as asked.
\end{proof}

We finish our work with two concluding statements.

\begin{theorem}
Let $R$ be a $2$-UNJ ring with $3 \in J(R)$, and $G$ a locally finite $3$-group. Then, $RG$ is a $2$-UNJ ring.
\end{theorem}

\begin{proof}
The proof is similar to that of the preceding Theorem \ref{T2}.
\end{proof}

\begin{theorem}
Let $R$ be a $2$-UNJ ring with $3 \in J(R)$, and $G$ a locally finite group of exponent $2$. Then, $RG$ is a $2$-UNJ ring.
\end{theorem}

\begin{proof}
Let $f \in RG$. Thus, there exists a finite subgroup $H$ of $G$ such that $f \in RH$. But, one knows that $H$ is a direct product of finite copies of $C_2$. However, since $2 \in U(R)$, $RH$ is a finite direct sum of copies of $RC_2$, and hence is a finite direct sum of copies of $R$, so $RH$ is 2-UNJ.

So, we can write $f = 1 + n + j$, where $n \in Nil(RH) \subseteq Nil(RG)$. But Lemma \ref{ext le} is applicable to have $J(RH) \subseteq J(RG)$, giving us that $j \in J(RH) \subseteq J(RG)$. Therefore, $RG$ is a 2-UNJ ring, as wanted.
\end{proof}


\end{document}